\documentclass[a4paper,11pt]{elsarticle}
\usepackage{amsthm,amsmath}
\usepackage{amssymb}
\usepackage{enumerate}

\newcommand\junk[1]{}

\newtheorem{theorem}{Theorem}[section]
\newtheorem{prop}[theorem]{Proposition}
\newtheorem{corollary}[theorem]{Corollary}
\newtheorem{definition}[theorem]{Definition}

\newtheorem{lemma}[theorem]{Lemma}

\DeclareMathOperator{\UP}{UP}
\DeclareMathOperator{\EX}{EX}

\newcommand{\BA}{\mathbf A}
\newcommand{\BB}{\mathbf B}
\newcommand{\BC}{\mathbf C}
\newcommand{\BD}{\mathbf D}

\newcommand{\BT}{\mathbf T}

\newcommand{\ig}{\mbox{\rm Inc} }
\newcommand{\Block}{\mbox{\rm Block} }

\newcommand{\trees}{\mathbb{T}}
\newcommand{\rtrees}{\mathbb{T}_{\mbox{\rm r}}}
\newcommand{\forests}{\mathbb{F}}
\newcommand{\rforests}{\mathbb{F}_{\mbox{\rm r}}}
\newcommand{\disu}{+}
\newcommand{\didi}{-}

\def\F{\mathbb{F}}

\def\T{\mathbb{T}}

\def\Ab{\mathbf A}
\def\Bb{\mathbf B}
\def\Cb{\mathbf C}
\def\Db{\mathbf D}
\def\Eb{\mathbf E}
\def\Tb{\mathbf T}
\def\Xb{\mathbf X}
\def\Yb{\mathbf Y}
\def\Zb{\mathbf Z}

\def\ig{\mbox{\rm Inc}}
\def\Block{\mbox{\rm Block}}
\makeatletter
\def\ps@pprintTitle{%
  \let\@oddhead\@empty
  \let\@evenhead\@empty
  \def\@oddfoot{\reset@font\hfil\thepage\hfil}
  \let\@evenfoot\@oddfoot
}
\makeatother

\begin{document}
\begin{frontmatter}
\title{Regular families of forests, antichains and duality pairs of relational structures}
\author[renyi]{P\'eter L. Erd\H os}
\author[elte]{D\"om\"ot\"or P\'alv\"olgyi\fnref{dome}}
\author[canada]{Claude Tardif}
\author[renyi]{G\'abor Tardos\fnref{tardos}}
\address[renyi]{Alfr\'ed R{\'e}nyi Institute of Mathematics, Re\'altanoda u 13-15 Budapest,
        1053 Hungary\\
        {\tt email}: $<$erdos.peter,tardos.gabor$>$@renyi.mta.hu}
\address[elte]{Institute of Mathematics, L.\ E\"otv\"os University,\\
        P\'azm\'any P\'eter s\'et\'any 1/C Budapest, 1117, Hungary\\
        {\tt email}: dom@cs.elte.hu}
\address[canada]{Royal Military College of Canada, PO Box 17000 Station ``Forces'' \\
        Kingston, Ontario, Canada, K7K 7B4\\
        {\tt email}: Claude.Tardif@rmc.ca}
\fntext[dome]{Research supported by Hungarian NSF (OTKA), grant PD 104386 and the J\'anos Bolyai Research Scholarship of the Hungarian Academy of Sciences.}
\fntext[tardos]{Research supported in part by the Cryptography ``Lend\"ulet'' project of the Hungarian Academy of Sciences.}

\begin{abstract}
Homomorphism duality pairs play a crucial role in the theory of
relational structures and in the Constraint Satisfaction Problem. The case
where both classes are finite is fully characterized. The case when both side
are infinite seems to be very complex. It is also known that no
finite-infinite duality pair is possible if we make the additional restriction
that both classes are antichains. In this paper we characterize the
infinite-finite antichain dualities and infinite-finite dualities with trees
or forest on the left hand side. This work builds on our earlier papers
\cite{elso} that gave several
examples of infinite-finite antichain duality pairs of directed graphs and
\cite{masodik} giving a complete characterization for
caterpillar dualities.
\end{abstract}
\end{frontmatter}

\section{Introduction}
A {\em homomorphism duality pair} is a couple
$(\mathcal{O}, \mathcal{D})$ where $\mathcal{O}$ and $\mathcal{D}$
are families of relational structures of the same type, such that
the following holds.
\begin{quote}
For any given relational structure $\BA$, there exists a homomorphism of $\BA$
to some member $\BD$ of $\mathcal{D}$ if and only if there is no homomorphism
of any member $\BT$ of $\mathcal{O}$ to $\BA$.
\end{quote}
We refer the reader to Section~\ref{prelim} for basic terminology on
relational structures and homomorphisms. Note that the simplest example is
that of directed graphs and arc-preserving maps.

Homomorphism duality pairs (duality pairs, for short) $(\mathcal{O}, \mathcal{D})$ where $\mathcal{D}$
consists of a single structure play a major role in the study of
constraint satisfaction problems (see~\cite{bkl}). Given a structure
$\BD$, conditions are sought that imply the existence of a duality pair
$(\mathcal{O}, \{ \BD \})$, where $\mathcal{O}$ satisfies given structural
properties. The structural properties considered in the literature
imply that the constraint satisfaction problem for $\BD$, that is,
the problem of determining whether an input structure admits a homomorphism
to $\BD$, can be solved by an efficient algorithm. In some cases the
type of duality considered gives further clues about the precise
descriptive complexity of the constraint satisfaction problem (see~\cite{bkl}).

A duality pair $(\mathcal{O}, \{ \BD \})$ is called
a {\em tree duality} if $\mathcal{O}$ consists of (relational) trees,
and a {\em finite duality} if $\mathcal{O}$ is finite.
The structures $\BD$ such that there exists a tree duality
$(\mathcal{O}, \{ \BD \})$ are characterized in \cite{fedvar},
and the structures $\BD$ such that there exists a finite duality
$(\mathcal{O}, \{ \BD \})$ are characterized in \cite{llt}.
In particular any finite duality $(\mathcal{O}, \{ \BD \})$
is essentially a tree duality (see~\cite{NT00}).

Finite dualities can be interpreted from an order theoretic point of view
in a category of relational structures preordered by the existence of homomorphisms.
Any finite duality can be reduced to a form $(\mathcal{O}, \{ \BD \})$,
where $|\mathcal{O}|$ is minimal (with respect to $\BD$).
In most cases $\BD$ does not admit a homomorphism to any member of $\mathcal{O}$.
The set $\mathcal{A} = \mathcal{O}\cup \{ \BD \}$ is then a ``maximal antichain''
in the sense that there exists no homomorphism between any two of its members,
and for any other structure $\BB$ of the same type, there is a structure $\BA$
in $\mathcal{A}$ such that there exists a homomorphism of $\BA$ to $\BB$ or
a homomorphism of $\BB$ to $\BA$. More precisely, for any structure $\BB$ of 
the same type, there is a structure $\BA$ in $\mathcal{O}$ such that there 
exists a homomorphism of $\BA$ to $\BB$,
or there is a structure $\BA$ in $\mathcal{D} = \{ \BD \}$  such that there exists a 
homomorphism of $\BB$ to $\BA$. An antichain $\mathcal{A}$ in 
an ordered set $\mathcal{P}$ is said to have the ``splitting property'' if 
it admits a partition into sets $\mathcal{O}$ and $\mathcal{D}$ such that 
for any element $\BB$, there exists $\BA \in \mathcal{O}$  such that $\BA \leq \BB$,
or there exists $\BA \in \mathcal{D}$  such that $\BB \leq \BA$.
(An antichain with the splitting property is necessarily maximal.)
When $\mathcal{P}$ is a category of relational structures and $\leq$ 
denotes the existence of a homomorphism, 
the antichains with the splitting property correspond to the
dualities $(\mathcal{O}, \mathcal{D})$ where $\mathcal{O} \cup \mathcal{D}$ 
is an antichain. Thus viewing dualities from an order-theoretic point of 
view leads to removing the emphasis on the case where  $\mathcal{D}$ 
is a singleton, and putting it instead on the case
where $\mathcal{O} \cup \mathcal{D}$ is a maximal antichain. 

In~\cite{FNT} it is shown how the finite duality pairs 
$(\mathcal{O}, \mathcal{D})$ are built up from finite duality pairs 
$(\mathcal{O}',  \mathcal{D}' )$ with $|\mathcal{D}'| = 1$. 
Furthermore the finite maximal antichains are shown to correspond 
to finite dualities, at least in the case where there is only 
one relation (e.g., the case of directed graphs); 
it is possible that this correspondence extends to all categories 
of relational structures.  
On the other hand the infinite maximal antichains are essentially 
unclassifiable, because any antichain can be greedily
extended to a maximal antichain (see~\cite{DENS,ES07}).
However the antichains with the splitting property have a lot of structure.
In~\cite{ES10} it is shown that in the case of directed graphs, 
such an antichain  $\mathcal{A} = \mathcal{O} \cup \mathcal{D}$
cannot have $\mathcal{O}$ finite and $\mathcal{D}$ infinite. 
The question of the existence of antichains with the splitting
property $\mathcal{A} = \mathcal{O} \cup \mathcal{D}$ with 
$\mathcal{O}$ infinite and $\mathcal{D}$ finite is
answered positively in~\cite{elso}. The context is again that of 
directed graphs, and it is shown that in any such infinite-finite 
antichain $\mathcal{A} = \mathcal{O} \cup \mathcal{D}$, 
$\mathcal{O}$ must consist of forests.
Thus the question arises as to which properties must an antichain 
$\mathcal{O}$ of forests satisfy for the existence
of a family $\mathcal{D}$ such that $\mathcal{A} = \mathcal{O} \cup \mathcal{D}$ 
is an infinite-finite antichain.
Note that this reverses the original question from the context of 
constraint satisfaction problems, where  conditions were sought 
on  $\mathcal{D} = \{ \BD \}$ for the existence of $\mathcal{O}$ such that
$(\mathcal{O},\mathcal{D})$ is a duality pair (with prescribed properties). 
In~\cite{elso} it is shown that if $\mathcal{O}$ is an antichain of digraph paths, 
then there can exist a finite $\mathcal{D}$ such that 
 $(\mathcal{O},\mathcal{D})$ is a duality only if $\mathcal{O}$ 
is ``regular'' in the sense of automata theory.

In \cite{masodik}, it is shown that regular languages can be used to characterize
the ``caterpillar dualities'' in general relational structures. Caterpillars are 
generalizations of paths, and the caterpillar dualities $(\mathcal{O},\{ \BD \})$
are of interest in the constraint satisfaction community as the dualities
for which $\mathcal{O}$ can be described in the ``smallest natural recursive 
fragment of Datalog'' (see \cite{CDK}).

In the present paper we extend the context of~\cite{elso} from digraphs to
general relational structures, and the context of~\cite{masodik} from
caterpillar dualities to general forest dualities. The criterion of regularity
remains relevant, but in the context of forest dualities it is necessary
to generalize it. Our notion of regular families of forests is similar
(but not identical) to other logics for trees surveyed in \cite{MB}. 

The paper is structured as follows. The next section is a brief
introduction to relational structures and homomorphisms.
In Section~\ref{sec:rff} we present regular families of forests
and prove that the family of forests that do not admit a homomorphism 
to a given structure is regular. In Section~\ref{sec:forest} we prove 
that for any regular family $\mathcal{O}$ of forests, there exists a finite 
family $\mathcal{D}$ of structures such that $(\mathcal{O},\mathcal{D})$ 
is a duality pair. In Sections \ref{sec:fo} and \ref{sec:mincore}
we deal with antichain dualities $(\mathcal{O},\mathcal{D})$
with $\mathcal{D}$ finite, and prove that these are
essentially forest dualities.
In Section \ref{sec:fo} we prove that 
if $(\mathcal{O},\mathcal{D})$ is a duality pair such 
that $\mathcal{D}$ is finite, $\mathcal{O}$ consists
of cores and $\mathcal{O} \cup \mathcal{D}$ 
is an antichain, then $\mathcal{O}$ consists of forests. 
Then Theorem~\ref{th:antichain} 
states that if $\mathcal{O}$ is an antichain of core forests,
then there exists a finite family $\mathcal{D}$ such that
$(\mathcal{O},\mathcal{D})$ is a duality pair if and only if
$\mathcal{O}$ is regular.
This uses Theorem~\ref{th:mincore} which states that the cores of the minimal elements of a regular set of forests form a regular set.
Section \ref{sec:mincore} is dedicated to the (quite technical) proof of this theorem.
In our results, specifying that
$\mathcal{O}$ consists of trees rather than forests
is equivalent to specifying that $\mathcal{D}$ consists of
a single structure.

\section{Preliminaries} \label{prelim}
\subsection{Relational structures}
A {\em type} is a finite set $\sigma = \{R_1,\dots,R_m\}$ of {\em relation symbols}, 
each with  an {\em arity} $r_i$ assigned to it. A $\sigma$-structure is a relational 
structure $\BA = \langle A;R_1(\BA),\dots,R_m(\BA)\rangle$ where $A$ is a finite set 
called the {\em universe} of $\BA$, and $R_i(\BA)$ is an $r_i$-ary relation on $A$ for each $i$.
 The elements of $R_i(\BA)$, $1\leq i \leq m$ will be called {\em hyperedges} of $\BA$. 
By analogy with the graph theoretic setting, the universe $A$ of $\BA$ will also be 
called its vertex-set, denoted $V(\BA)$.

A $\sigma$-structure $\BA$ may be described by its bipartite 
{\em incidence multigraph} $\ig(\BA)$ defined as follows. 
The two parts of $\ig(\BA)$ are $V(\BA)$ and $\Block(\BA)$, where
$$
\Block(\BA) = \{ (R,(x_1, \ldots, x_{r})) : R \in \sigma \mbox{ has arity $r$ and }
(x_1, \ldots, x_{r}) \in R(\BA) \},
$$
and with edges $e_{a,i,B}$ joining $a \in V(\BA)$ to $B=(R,(x_1, \ldots, x_{r})) \in \Block(\BA)$ when $x_i = a$. Thus, the degree of $B=(R,(x_1, \ldots, x_{r}))$ in $\ig(\BA)$
is precisely $r$. Here ``degree'' means number of incident edges rather than number of neighbors because parallel edges are possible: If $x_i = x_j = a \in V(\BA)$, then $e_{a,i,B}$ and 
$e_{a,j,B}$ both join $a$ and $B$.
A $\sigma$-structure $\BA$ is called a {\em $\sigma$-tree} (or {\em tree} for short) if $\ig(\BA)$ is a (graph-theoretic) tree, that is, it is connected and has no cycles or parallel edges. 
Similarly, $\BA$ is called a {\em $\sigma$-forest} (or {\em forest} for short)
if $\ig(\BT)$ is a (graph-theoretic) forest, that is, it has no cycles or
parallel edges.

A $\sigma$-structure $\Ab$ is a {\em
substructure} of $\Bb$ (in notation: $\Ab\subseteq\Bb$) if the
universe and relations of $\Ab$ are
subsets of the corresponding families for $\Bb$. In this case $\ig(\Ab)$ is a
subgraph of $\ig(\Bb)$. In particular the components of $\ig(\Bb)$ determine
substructures that we call the {\em components} of $\Bb$. 

\subsection{Homomorphisms}
For $\sigma$-structures $\BA$ and $\BB$, a {\em homomorphism} from $\BA$ to $\BB$ is a map $f: V(\BA) \to V(\BB)$ such that $f(R_i(\BA)) \subseteq R_i(\BB)$ for all $i=1,\dots,m$, where for any $r$-ary relation $R \in \sigma$ we have
$$
f(R) = \{(f(x_1),\dots,f(x_r)):(x_1,\dots,x_r) \in R\}.
$$
We write $\BA \rightarrow \BB$ if there exists a homomorphism from $\BA$ to $\BB$, and $\BA \not \rightarrow \BB$ otherwise. We write $\BA \leftrightarrow \BB$ when $\BA \rightarrow \BB$ and $\BB \rightarrow \BA$; $\BA$ and $\BB$ are then called {\em homomorphically equivalent}. For a finite structure $\BA$, we can always find a structure $\BB$ such that $\BA \leftrightarrow \BB$ and the cardinality of $V(\BB)$ is minimal with respect to this property. It is well known (see 
\cite{NT00}) that such $\BB$ is unique up to isomorphism. We then call 
$\BB$ the {\em core} of $\BA$.


Recall the definition of homomorphism duality pairs from the beginning of the
Introduction. In our definition of $\sigma$-structure we insisted that it must
be finite but in this paragraph we comment on the natural extension to
infinite structures. The families
$\mathcal O$ and $\mathcal D$ in the duality
pairs $(\mathcal O,\mathcal D)$ we consider in this paper consist of finite
structures and their defining condition is required to hold for any finite
structure $\BA$. For general duality pairs this does not imply that the same
condition also holds for infinite structures $\BA$. But in this paper we
consider duality pairs $(\mathcal O,\mathcal D)$ with $\mathcal D$ finite. For
a finite family $\mathcal D$ of finite structures, a compactness argument
shows that if an infinite structure $\BA$ has no homomorphism to any member of
$\mathcal D$, then this also holds for a finite substructure of $\BA$. This
implies that the duality pairs we consider in this paper are duality pairs even
if considered in the context of arbitrary (not necessarily finite) structures.

\section{Regular families of forests} \label{sec:rff}

Let $\trees$ be the set of $\sigma$-trees and
$\forests$ the set of $\sigma$-forests. 
A {\em rooted} $\sigma$-structure $(\BA,a)$ is a $\sigma$-structure $\BA$ with an 
arbitrary element $a \in V(\BA)$ designated as the root. 
Let $\rforests$ and $\rtrees$ denote respectively
the set of rooted $\sigma$-forests and of rooted $\sigma$-trees. We do not
distinguish isomorphic structures, so more precisely $\trees$, $\forests$,
$\rtrees$ and $\rforests$ are the set of 
{\em isomorphism classes} of trees, forests and rooted trees and rooted
forests, respectively.
We introduce two operations:
\begin{itemize}
\item {\bf unrooting}\quad  $[ \cdot ] : \rforests \to \forests$  forgets the root;
\item {\bf combining} \quad $+ : \rforests \times \rforests \to \rforests $ 
takes the disjoint union of two rooted forests and identifies their roots.
\end{itemize}

Note that the single-vertex rooted tree with empty relations is the identity
of the combining operation. We denote it by $\Tb_0$.

\begin{definition}\label{rffdef} {\rm Let $\mathcal{O}$ be a subset of $\forests$.
\begin{itemize}
\item For $(\BA,a) \in \rforests$, the set 
$\mathcal{O} - (\BA,a) \subseteq \rforests$ is defined by
$$
\mathcal{O} - (\BA,a) 
= \{ (\BB,b) \in \rforests\mid[(\BA,a)+(\BB,b)]\in \mathcal{O} \}.
$$
\item The equivalence $\sim_{\mathcal{O}}$ on $\rforests$ is defined by
$(\BA,a) \sim_{\mathcal{O}} (\BA',a')$ if 
$\mathcal{O} - (\BA,a) =   \mathcal{O} - (\BA',a')$.
\item $\mathcal{O}$ is called {\em regular} if $\rforests$ 
has only a finite number of equivalence classes under $\sim_{\mathcal{O}}$.
\end{itemize} }
\end{definition}

Given a set $\mathcal{O} \subseteq \forests$, we can define a graph $G$
whose vertex-set is $\rforests$ and whose edges join pairs $(\BA,a)$,
$(\BB,b)$ such that $[(\BA,a) + (\BB,b)] \in \mathcal{O}$.
Then $\mathcal{O} - (\BA,a)$ is the neighborhood of $(\BA,a)$ in $G$.
Thus $\mathcal O$ is regular if and only if this graph can be obtained from a
finite graph by blowing up its vertices. 

In~\cite{elso} and \cite{masodik}, the natural description of
paths and caterpillars in terms of words over an alphabet
allowed a more direct correspondence between regular languages
and regular families of paths or caterpillars. Definition~\ref{rffdef}
is in the spirit of the Myhill-Nerode theorem, which
states that a language $\mathcal{L}$ over an alphabet $\Sigma$ 
is regular if and only if the infinitely many words $a \in \Sigma^*$ 
define only finitely many distinct
extension sets $\mathcal{L} - a =\{b\in\Sigma^*\mid ab\in \mathcal{L}\}$.
In Definition~\ref{rffdef}, $[ \cdot + \cdot ]$ plays the 
role of concatenation and the sets $\mathcal{O} - (\BA,x)$
play the role of extension sets.

The following result establishes basic properties of regular sets
of forests.
\begin{lemma} \label{lm:bprl}
Let $\mathcal{O}_1$ and $\mathcal{O}_2$ be regular
subsets of $\forests$. Then $\mathcal{O}_1 \cup \mathcal{O}_2$,
$\mathcal{O}_1 \cap \mathcal{O}_2$ and $\forests \setminus \mathcal{O}_1$
are also regular.
\end{lemma}
\begin{proof}
This follows from the fact that for any $(\BA,x) \in \rforests$ we have
\begin{eqnarray*}
(\mathcal{O}_1 \cup \mathcal{O}_2) - (\BA,x) & = & 
(\mathcal{O}_1 - (\BA,x)) \cup (\mathcal{O}_2 - (\BA,x)), \\
(\mathcal{O}_1 \cap \mathcal{O}_2) - (\BA,x) & = & 
(\mathcal{O}_1 - (\BA,x)) \cap (\mathcal{O}_2 - (\BA,x)) \mbox{ and} \\
(\forests \setminus \mathcal{O}_1) - (\BA,x) & = & 
\rforests \setminus (\mathcal{O}_1 - (\BA,x)). 
\end{eqnarray*}
Therefore 
\begin{eqnarray*}
|\rforests / \sim_{\mathcal{O}_1 \cup \mathcal{O}_2}| & \leq &
|\rforests / \sim_{\mathcal{O}_1}| \cdot |\rforests / \sim_{\mathcal{O}_2}|, \\
|\rforests / \sim_{\mathcal{O}_1 \cap \mathcal{O}_2}| & \leq &
|\rforests / \sim_{\mathcal{O}_1}| \cdot |\rforests / \sim_{\mathcal{O}_2}| \mbox{ and} \\
|\rforests / \sim_{\forests \setminus \mathcal{O}_1}| & = &
|\rforests / \sim_{\mathcal{O}_1}|.
\end{eqnarray*}
\end{proof}

\begin{lemma} \label{lm:nda}
For any $\sigma$-structure $\BD$, the family 
$\mathcal{H}_{\BD}=\{\BA\in \forests\mid\BA \to \BD\}$ is regular.
\end{lemma}
\begin{proof}
For any vertex $z\in V(\BD)$ let $\mathcal{S}_z$ be the set of the rooted
forests $(\BB,y)$ such that there exists a homomorphism 
$f:\BB \to \BD$ with $f(y) = z$. Then for any
$(\BA,x) \in \rforests$, we have
$$\mathcal{H}_{\BD} - (\BA,x) = \bigcup \{ \mathcal{S}_z\mid(\BA,x) \in \mathcal{S}_z\}.$$ 
Therefore $|\rforests / \sim_{\mathcal{H}_{\BD}}| \leq 2^{|V(\BD)|}$.
\end{proof}

\begin{corollary} \label{lm:obreg}
Let $\mathcal{D}$ be a finite family of $\sigma$-structures. Then 
the family $\mathcal{O}_{\mathcal{D}}$ of forests which do not admit a 
homomorphism to any structure in $\mathcal{D}$ is regular.
\end{corollary}
\begin{proof}
We have $\mathcal{O}_{\mathcal{D}} = 
\bigcap_{\BD \in \mathcal{D}}(\forests \setminus \mathcal{H}_{\BD})$,
and the latter is regular by Lemmas~\ref{lm:nda} and~\ref{lm:bprl}. 
\end{proof}

\begin{corollary} \label{lm:tobreg}
Let $\BD$ be a $\sigma$-structure.
Then the family of trees which do not admit a homomorphism to $\BD$ is regular.
\end{corollary}

\begin{proof}
The family $\trees$ of $\sigma$-trees is regular, since
$\rforests / \sim_{\trees} = \{\rtrees,\rforests\setminus\rtrees\}$.
Thus for any $\sigma$-structure $\BD$, $\mathcal{H}_{\BD} \cap \trees$
and $\mathcal{O}_{\{\BD\}} \cap \trees$ are regular.
\end{proof}
Note that by a ``regular family of trees'' we mean a ``regular family of 
forests whose members are trees.'' A second interpretation is possible, 
obtained by replacing $\forests$ and $\rforests$ by $\trees$ and $\rtrees$ 
in Definition~\ref{rffdef}. However the two interpretations
turn out to be equivalent, since for $\mathcal{O} \subseteq \trees$, we have 
$\mathcal{O} - (\BA,x) = \emptyset$ for all $(\BA,x) \in \rforests \setminus \rtrees$.
Thus there is no need for a separate definition of ``regular family of trees.''

How about ``regular families of rooted forests''?
One might be tempted to define these in a similar fashion, 
by defining the extension sets of $\mathcal{O}_r\in \rforests$ by
$$
\mathcal{O}_r - (\BA,a) 
= \{ (\BB,b) \in \rforests\mid(\BA,a)+(\BB,b)\in \mathcal{O}_r \}.
$$
The regular families of rooted forests would then be those with finitely 
many extension sets. However, our regular families of forests are not simply the
unrootings of such ``regular families of rooted forests.''
Indeed, consider any family $\mathcal{O}_r$ that consists only of rooted trees 
whose root has degree exactly one in the incidence multigraph. With the above
definition, such a family becomes regular, with $\mathcal{O}_r$, $\{\Tb_0\}$ 
and $\emptyset$ being the only possible extension sets.
Thus this approach leads to uncountably many families of rooted forests
that are regular for trivial reasons, and whose unrootings yield
uncountably many families of forests.

In contrast, our last result of this section will show that just like in the case
of regular languages over any given alphabet, there are countably many regular families 
of $\sigma$-forests over any given type $\sigma$.
\begin{definition} \label{concatdef} {\rm
Let $R \in \sigma$ be a relation of arity $r$, and
$(\BA_1,x_1), \ldots, (\BA_r,x_r)$ rooted $\sigma$-structures.
The {\em concatenation} $\BC(R,(\BA_1,x_1), \ldots, (\BA_r,x_r))$
is the $\sigma$-structure $\BC$ obtained from the disjoint union
of $\BA_1, \ldots, \BA_r$ by adding $(x_1, \ldots, x_r)$ to $R(\BC)$.
} \end{definition}
Note that the concatenation of rooted $\sigma$-trees is a $\sigma$-tree.
\begin{theorem}\label{th:infregfam} There are countably many regular families
of $\sigma$-forests.
\end{theorem}

\begin{proof} We associate a type $o(\sigma)$ of operations to the type $\sigma$
as follows. For each $R \in \sigma$ of arity $r$, $o(\sigma)$ contains
$r$ operations $\mu_{R,1}, \ldots, \mu_{R,r}$ of arity $r$. In addition, 
$o(\sigma)$ contains an operation $\nu$ of arity $1$.
We define the algebra $(\rforests, o(\sigma))$ as follows.
\begin{itemize}
\item $\nu(\BA,x)$ is the rooted forest obtained by adding an isolated element to
$V(\BA)$, which becomes the new root.
\item For $R \in \sigma$ of arity $r$ and $i \in \{ 1, \ldots, r\}$, $\mu_{R,i}$
is defined by
$$\mu_{R,i}((\BA_1,x_1), \ldots, (\BA_r,x_r))
= (\BC(R,(\BA_1,x_1), \ldots, (\BA_r,x_r)),x_i).$$
\end{itemize}
Note that the whole of $\rforests$ is generated by $\{\Tb_0\}$.

Let $\mathcal{O}$ be a family of forests. We show that $\sim_{\mathcal{O}}$ is a congruence
on $(\rforests, o(\sigma))$. Indeed if $(\BA,x) \sim_{\mathcal{O}} (\BA',x')$
and $[\nu(\BA,x) + (\BB,y)] \in \mathcal{O}$, then since 
$[\nu(\BA,x) + (\BB,y)] = [(\BA,x) + \nu(\BB,y)]$, we have
$[\nu(\BA',x') + (\BB,y)] = [(\BA',x') + \nu(\BB,y)] \in \mathcal{O}$.
Now suppose that $R \in \sigma$ has arity $r$, $i, j \in \{ 1, \ldots r\}$
and $(\BA_i,x_i) \sim_{\mathcal{O}} (\BA_i',x_i')$. We show that 
for any $(\BA_k,x_k) = (\BA_k',x_k')$ and $k  \in \{1, \ldots, r\} \setminus \{i\}$, we have
$\mu_{R,j}((\BA_1,x_1), \ldots, (\BA_r,x_r)) \sim_{\mathcal{O}} 
\mu_{R,j}((\BA_1',x_1'), \ldots, (\BA_r',x_r'))$. Suppose that 
$[\mu_{R,j}((\BA_1,x_1), \ldots, (\BA_r,x_r)) + (\BB,y)] \in \mathcal{O}$.
We define  $(\BC_1,z_1), \ldots, (\BC_r,z_r)$ by 
$\BC_k = (\BA_k,x_k)$ if $k\notin\{i,j\}$; $\BC_i = \{\Tb_0\}$,
$\BC_j=(\BA_j,x_j)+(\BB,y)$ if $j \neq i$; finally $\BC_i = (\BB,y)$ if $j=i$.
Then $[\mu_{R,j}((\BA_1,x_1), \ldots, (\BA_r,x_r)) + (\BB,y)]\!\!=\!\! 
[(\BA_i,x_i) + \mu_{R,i}((\BC_1,x_1), \ldots, (\BC_r,x_r))]$.
As 
$(\BA_i,x_i) \sim_{\mathcal{O}} (\BA_i',x_i')$, we then have
$[(\BA_i',x_i') + \mu_{R,i}((\BC_1,x_1), \ldots, (\BC_r,x_r))]$ 
 $= [\mu_{R,j}((\BA_1',x_1'), \ldots, (\BA_r',x_r')) + (\BB,y)] \in \mathcal{O}$.
Thus $\sim_{\mathcal{O}}$ is a congruence on $(\rforests, o(\sigma))$.

If $\mathcal{O}$ is regular, then  $(\rforests, o(\sigma))
/\!\!\sim_{\mathcal{O}}$  is finite, thus $\sim_{\mathcal{O}}$ is the kernel
of a homomorphism
$\phi(\rforests, o(\sigma))\to (X, o(\sigma))$, where $(X, o(\sigma))$
is a finite $o(\sigma)$-algebra. Since $\rforests$ is generated by $\{\Tb_0\}$,
$\phi$ is completely determined by $\phi(\Tb_0)$. Take the unrootings of
rooted forests in an $\sim_{\mathcal O}$ equivalence class. We have
$\BA=[(\BA,x)+\Tb_0]$, showing that either all of these forests belong to
$\mathcal O$ or none of them. To emphasize the
similarity with the finite automaton characterization of regular languages,
we call the elements of $X$ {\em states}, $\phi(\Tb_0)$ the {\em initial state}
and the states corresponding to the classes of rooted forests whose unrootings
are in $\mathcal O$ the {\em terminal states}.

There are finitely many $o(\sigma)$-algebras on any finite set, finitely many
choices for the initial state and for the terminal states. The family
$\mathcal O$ is determined by these choices (except whether the empty
structure belongs to it), thus the number of regular families of forests is
countable.
\end{proof}

\section{Constructions of duals}\label{sec:forest}

\subsection{Duals of families of trees} 

\begin{definition} \label{def:dualtree} {\rm
Let $\mathcal{O}$ be a regular family of trees. We define
the structure $\BD(\mathcal{O})$ as follows.

The vertices of $\BD(\mathcal{O})$ are the sets 
$\mathcal{V} \subseteq \rtrees$ satisfying the following properties.
\begin{itemize}
\item[(i)] $\Tb_0\in\mathcal{V}$.
\item[(ii)] If $(\BA,a) \in \mathcal{V}$, then $\BA \not \in \mathcal{O}$.
\item[(iii)] $\mathcal{V} =
\rtrees \setminus\bigcup_{(\BA,a)\in\mathcal H} (\mathcal{O} -(\BA,a))$ for
some family $\mathcal H$ of rooted trees.
\end{itemize} 

For $R \in \sigma$ of arity $r$, we have 
$(\mathcal{V}_1, \ldots, \mathcal{V}_r) \in R(\BD(\mathcal{O}))$ if 
for each $(\BA_i,a_i) \in \mathcal{V}_i$, $i = 1, \ldots, r$
and $\BT = \BC(R,((\BA_1,a_1), \ldots, (\BA_r,a_r)))$,
we have $(\BT,a_j) \in \mathcal{V}_j$,
$j = 1, \ldots, r$, where $\BC$ is the concatenation of Definition~\ref{concatdef}.
} \end{definition}

Note that since $\mathcal{O}$ is regular, there are finitely many sets
of the form $\mathcal{O} - (\BA,a)$,
hence $V(\BD(\mathcal{O}))$ is finite as required.

\begin{lemma}\label{lem:tinimage}
Let $\BA$ be a $\sigma$-tree and $\phi: \BA \to \BD(\mathcal{O})$
a homomorphism.
Then for every $a \in V(\BA)$, $(\BA,a) \in \phi(a)$.
\end{lemma}
\begin{proof}
We use induction on the number $|\bigcup_{R \in \sigma} R(\BA)|$ of
hyperedges of $\BA$. If $\BA$ has no hyperedges, the result follows from
item (i) of Definition \ref{def:dualtree}. Now suppose that
the result is valid for any tree with fewer hyperedges than $\BA$.
Let $(x_1, \ldots, x_r) \in R(\BA)$ be a hyperedge such that 
$a \in \{x_1, \ldots, x_r\}$. Then $\BA = \BC(R,(\BA_1,x_1), \ldots, (\BA_r,x_r))$,
where $\BA_1, \ldots, \BA_r$ are the components of the
forest obtained from removing $(x_1, \ldots, x_r)$ from $R(\BA)$. The restriction
of $\phi$ to each $\BA_i$ is a homomorphism, so by the induction hypothesis,
$(\BA_i,x_i) \in \phi(x_i)$ for $i = 1, \ldots, r$. Since $\phi$ is a homomorphism
and $(x_1, \ldots, x_r) \in R(\BA)$, we then have 
$(\BA,x_j) \in \phi(x_j)$
for $j = 1, \ldots, r$, and in particular,  $(\BA,a) \in \phi(a)$.
\end{proof}

\begin{theorem}\label{th:tree-dual}
Let $\mathcal{O}$ be a regular family of trees. Then for any 
$\sigma$-structure $\BB$, there exists a homomorphism of
$\BB$ to $\BD(\mathcal{O})$ if and only if no tree in
$\mathcal{O}$ admits a homomorphism to $\BB$.
\end{theorem}

\begin{proof}
Let $\BB$ be a $\sigma$-structure that admits a homomorphism $\phi$ to $\BD(\mathcal{O})$.
Suppose that there exists a tree $\BA$ in $\mathcal{O}$ that admits a homomorphism
$\psi$ to $\BB$. Then $\phi \circ \psi: \BA \to \BD(\mathcal{O})$ is a homomorphism.
Therefore by Lemma~\ref{lem:tinimage}, for any $a \in V(\BA)$,
$(\BA,a) \in \phi(\psi(a))$. However this contradicts
item (ii) in Definition~\ref{def:dualtree}.
Thus no element of $\mathcal{O}$ admits a homomorphism to $\BB$.

Conversely, suppose that no element of $\mathcal{O}$ admits a homomorphism to $\BB$.
For a vertex $b$ of $\BB$, we define the following sets
\begin{itemize}
\item $S(b) = \{ (\BA,a) \in \rtrees\mid\mbox{ there exists a homomorphism
  }\phi: \BA \to \BB$ such that $\phi(a) = b\;\}$,
\item $\phi(b)=\rtrees\setminus\bigcup(\mathcal O-(\BA,a))$, where the union
  is taken for all $(\BA,a)\in\rtrees$ with 
$S(b) \cap (\mathcal{O} - (\BA,a)) = \emptyset$.
\end{itemize}

Note that $S(b)\subseteq\phi(b)$ for all vertex $b$, furthermore $\phi(b)$ is
the minimal set containing $S(b)$ that is in the form required by item (iii)
of Definition~\ref{def:dualtree}. Item (i) is also satisfied by $\phi(b)$ as
$\Tb_0\in S(b)\subseteq\phi(b)$. Since no element of $\mathcal O$ admits a
homomorphism to $\BB$, we have $\BA\notin\mathcal O$ whenever $(\BA,a)\in
S(b)$. This implies $S(b)\cap(\mathcal O-\Tb_0)=\emptyset$, thus
$\phi(b)\subseteq \rtrees\setminus(\mathcal O-\Tb_0)$. Therefore $\phi(b)$ also
satisfies item (ii) and thus it is a vertex of $\BD(\mathcal O)$. We will show
that the map $\phi:V(\BB)\to V(\BD(\mathcal O))$ is a homomorphism.

Let $R \in \sigma$ be a relation of arity $r$ and 
$(b_1, \ldots, b_r) \in R(\BB)$. We need to show that
$(\phi(b_1), \ldots, \phi(b_r)) \in R(\BD(\mathcal{O}))$,
that is, for every $(\BA_i,a_i) \in \phi(b_i)$, $i = 1, \ldots, r$
and $\BT = \BC(R,(\BA_1,a_1), \ldots, (\BA_r,a_r))$
we have $(\BT,a_j) \in \phi(b_j)$, $j = 1, \ldots, r$. 
We will proceed by contradiction, supposing that there exists
an index $j$ such that $(\BT,a_j) \not \in \phi(b_j)$.
First note that we cannot then have $(\BA_i,a_i) \in S(b_i)$, $i = 1, \ldots, r$.
For otherwise the corresponding homomorphisms $\psi_i: \BA_i \to \BB$
with $\psi_i(a_i) = b_i$, $i = 1, \ldots, r$ could be combined
into a homomorphism $\psi: \BT \to \BB$ such that 
$\psi(a_j) = b_j$, giving $(\BT,a_j) \in S(b_j) \subseteq \phi(b_j)$. 
Thus we can assume that there exists at least one index $i$ such that 
$(\BA_i,a_i) \in \phi(b_i) \setminus S(b_i)$.
(We could have $i = j$.) 
We will suppose that the cardinality 
$|\{ k \mid (\BA_k,a_k) \in \phi(b_k) \setminus S(b_k)\}|$ 
is as small as possible, and derive a contradiction.

Since $(\BT, a_j) \not \in \phi(b_j)$, there exists $(\BA_j',a_j')$
such that $S(b_j) \cap (\mathcal{O} - (\BA_j',a_j')) = \emptyset$
and $(\BT, a_j) \in (\mathcal{O} - (\BA_j',a_j'))$.
We then have $[(\BA_j',a_j') + (\BT, a_j)] \in \mathcal{O}$. 
We can rewrite the latter tree as a different sum by moving the root:
$[(\BA_j',a_j') + (\BT, a_j)] = [(\BA_i,a_i) + (\BT', a_i)]$;
this implicitly characterizes $\BT'$.
We then have $(\BA_i,a_i) \in (\mathcal{O} - (\BT', a_i))$.
Thus $(\BA_i,a_i) \in \phi(b_i) \not \subseteq \rtrees \setminus
(\mathcal{O} - (\BT'', a_i''))$. By the definition of
$\phi(b_i)$, this means that there exists $(\BA_i'',a_i'') 
\in S(b_i) \cap (\mathcal{O} - (\BT', a_i))$.

For $k \neq i$, put $(\BA_k'',a_k'') = (\BA_k,a_k)$.
The tree $\BT'' = \BC(R,(\BA_1'',a_1''), \ldots,$ $(\BA_r'',a_r''))$ 
is obtained by replacing $(\BA_i,a_i)$ by $(\BA_i'',a_i'')$
in the concatenation defining $\BT$. 
By the minimality of $|\{ k\mid (\BA_k,a_k) \in \phi(b_k) \setminus S(b_k)\}|$,
we have $(\BT'',a_k'') \in \phi(b_k)$, $k = 1, \ldots, r$,
and in particular $(\BT'',a_j'') \in \phi(b_j)$.
A second argument will also show  $(\BT'',a_j'') \not \in \phi(b_j)$:
Since $S(b_j) \cap (\mathcal{O} - (\BA_j',a_j')) = \emptyset$,
we have $\phi(b_j) \subseteq \rtrees \setminus (\mathcal{O} - (\BA_j',a_j'))$.
But by moving the root in $[(\BT'',a_j'') + (\BA_j',a_j')]$ we get
$$[(\BT'',a_j'') + (\BA_j',a_j')]
= [(\BT', a_i) + (\BA_i'',a_i'')] \in \mathcal{O},$$
since $(\BA_i'',a_i'') \in \mathcal{O} - (\BT', a_i)$.
Thus $(\BT'',a_j'') \in (\mathcal{O} - (\BA_j',a_j'))$,
which is disjoint from $\phi(b_j)$. Therefore $(\BT'',a_j'') \not \in \phi(b_j)$,
a contradiction.
\end{proof}

\subsection{Duals of forests}\label{dualforest}

We will use the symbol $\disu$ to denote the disjoint union of structures
including forests (without roots). Accordingly, for $\mathcal{O} \subseteq \forests$,
we put
$$\mathcal{O} \didi \BA = 
\{ \BB \in \forests\mid\BA \disu \BB \in \mathcal{O}\}.$$
Note that we use $+$ and $-$ also in the context of rooted forests for a related but different notion.

The equivalence $\approx_{\mathcal{O}}$ on $\forests$ is defined by
$$
\BA \approx_{\mathcal{O}} \BA' \mbox{ if } 
\mathcal{O} \didi \BA =  \mathcal{O} \didi \BA'.
$$
\begin{lemma}\label{lem:tree-eq}
If\/ $\mathcal{O} \subseteq \forests$ is regular, then there are finitely many $\approx_{\mathcal{O}}$-equivalence classes,
and these are regular.
\end{lemma}
\begin{proof} 
Given a forest $\BA$, let $(\BA^+,x)$ be the rooted forest obtained by
adding an isolated element $x$ to $V(\BA)$ and designating it as the root.
Then
$$
\mathcal{O} \didi \BA = \{ [(\BB,b)] \in \forests\mid(\BB,b) \in \mathcal{O} - (\BA^+,x) \}.
$$
Since $\mathcal{O}$ is regular, there are finitely many sets of the form
$\mathcal{O} - (\BA^+,x)$, hence finitely many sets of the form $\mathcal{O}
\didi \BA$ and finitely many $\approx_{\mathcal O}$ equivalence classes.

Now for $\BA \in \forests$ and $(\BB,b), (\BC,c) \in \rforests$
we have $\BA \disu [(\BB,b) + (\BC,c)] = [(\BA \disu \BB,b) + (\BC,c)]$,
Thus $(\mathcal{O} \didi \BA) - (\BB,b) = \mathcal{O} - (\BA \disu \BB,b)$.
If $\mathcal{O}$ is regular, there are finitely many sets of the form
$\mathcal{O} - (\BA \disu \BB,b)$, that is, finitely many sets of the 
form $(\mathcal{O} \didi \BA) - (\BB,b)$. Thus $\mathcal{O} \didi \BA$
is regular.

Any $\approx_{\mathcal O}$ equivalence class can be obtained as the set of
forests $\BA$ satisfying finitely many conditions of the form
$\BB_i\in\mathcal O\didi\BA$ or $\BB_i\notin\mathcal O\didi\BA$. Since
$\BB_i\in\mathcal O\didi\BA$ if and only if $\BA\in\mathcal O\didi\BB_i$ the
equivalence class can be obtained as the intersection of finitely many regular
sets or their complements, and by Lemma~\ref{lm:bprl} it must be regular.
\end{proof}

\begin{definition} \label{def:adm} {\rm
Let $\mathcal{O}$ be a regular family of forests.
A family $\mathcal{T} \subset \trees$ is called {\em $\mathcal{O}$-admissible}
if it satisfies the following properties.
\begin{itemize}
\item $\mathcal{T}$ contains a connected component of each element of $\mathcal{O}$.
\item If $\BA \in \mathcal{T}$ and $\BA' \approx_{\mathcal{O}} \BA$, then
$\BA' \in \mathcal{T}$.
\end{itemize} 
} \end{definition}
By Lemma~\ref{lem:tree-eq}, each $\approx_{\mathcal{O}}$-equivalence class
$\BA /\!\!\approx_{\mathcal{O}}$ is regular. Since $\trees$ is also regular,
$\trees \cap (\BA /\!\!\approx_{\mathcal{O}})$ is regular. An $\mathcal{O}$-admissible family
$\mathcal{T}$ is a finite union of such regular families (since 
$\forests /\!\!\approx_{\mathcal{O}}$ is finite), hence it is regular.
Thus by Theorem~\ref{th:tree-dual}, $\mathcal{T}$ admits a dual $\BD(\mathcal{T})$.

\begin{definition} \label{def:dualforests} {\rm
Let $\mathcal{O}$ be a regular family of forests. The family 
$\mathcal{D}(\mathcal{O})$ is defined as the set of structures
$\BD(\mathcal{T})$ such that $\mathcal{T}$ is an
$\mathcal{O}$-admissible family of trees.
} \end{definition}

\begin{theorem} \label{th:dualforests}
Let $\mathcal{O}$ be a regular family of forests.
Then $\mathcal{D}(\mathcal{O})$ is finite, and for any 
$\sigma$-structure $\BB$, there exists a homomorphism of $\BB$ to some
member of $\mathcal{D}(\mathcal{O})$ if and only if no member of $\mathcal{O}$
admits a homomorphism to $\BB$.
\end{theorem}
\begin{proof}
The number of elements of $\mathcal{D}(\mathcal{O})$ is at most
$2^{|\forests / \approx_{\mathcal{O}}|}$ since every
$\mathcal{O}$-admissible family is a union of sets of 
the form $\trees \cap (\BA /\!\!\approx_{\mathcal{O}})$.
By Lemma~\ref{lem:tree-eq}, $|\forests /\!\!\approx_{\mathcal{O}}|$
is finite, hence $|\mathcal{D}(\mathcal{O})|$ is finite.

For $\BA \in \mathcal{O}$ and $\BD(\mathcal{T}) \in \mathcal{D}(\mathcal{O})$,
there exists a component $\BT$ of $\BA$ which belongs to $\mathcal{T}$.
We have $\BT \not \rightarrow \BD(\mathcal{T})$ whence 
$\BA \not \rightarrow \BD(\mathcal{T})$. Thus if a $\sigma$-structure
$\BB$ admits a homomorphism from some $\BA \in \mathcal{O}$, then there is
no homomorphism of $\BB$ to any member of $\mathcal{D}(\mathcal{O})$.

Now let $\BB$ be a $\sigma$-structure such that no member of
$\mathcal{O}$ admits a homomorphism to $\BB$.
Put $\mathcal{S}(B) = \{ \BT \in \trees\mid\BT \not \rightarrow \BB\}$
and 
$$\mathcal{T}(B) = \{ \BT \in \trees\mid (\BT /\!\!\approx_{\mathcal{O}})
\subseteq  \mathcal{S}(B)\}.$$
 Then $\mathcal{S}(B)$ contains a component of
every $\BA$ in $\mathcal{O}$. We need to show that the same holds for
$\mathcal{T}(B)$. 

Suppose for a contradiction that some $\BA \in \mathcal{O}$
has no component $\BT \in \mathcal{S}(B)$ such that 
$(\BT /\!\!\approx_{\mathcal{O}}) \subseteq  \mathcal{S}(B)$.
Let $\BT_1, \ldots, \BT_k$ be the components of $\BA$ in $\mathcal{S}(B)$
such that there exists $\BT_i'$ with $\BT_i' \rightarrow \BB$
and $\BT_i' \approx_{\mathcal{O}} \BT_i$, $i = 1, \ldots, k$. 
Put $\BA_0 = \BA$, and for $i = 1, \ldots, k$, let $\BA_i$ be the
forest obtained from $\BA_{i-1}$ by replacing the component $\BT_i$
by $\BT_i'$. Then by definition of $\approx_{\mathcal{O}}$,
we have  $\BA_i \in \mathcal{O}$ for all $i$, hence $\BA_k \in \mathcal{O}$.
However, $\BA_k \rightarrow \BB$, which contradicts the fact that no member of
$\mathcal{O}$ admits a homomorphism to $\BB$.

Thus $\mathcal{T}(B)$ is a union of $\approx_{\mathcal{O}}$ classes which contains 
a component of every $\BA$ in $\mathcal{O}$. Thus by definition,
$\mathcal{T}(B)$ is $\mathcal{O}$-admissible, and 
$\BD(\mathcal{T}(B)) \in \mathcal{D}(\mathcal{O})$.
No member of $\mathcal{T}(B)$ admits a homomorphism to $\BB$,
thus $\BB$ admits a homomorphism to $\BD(\mathcal{T}(B))$.
\end{proof}

\section{Antichain dualities and forests} \label{sec:fo}
A duality pair $(\mathcal{O}, \mathcal{D})$ is called an 
{\em antichain duality} if $\mathcal{O} \cup \mathcal{D}$ is an antichain.
We will suppose that $\mathcal{O}$ consists of cores, since any element
of $\mathcal{O}$ can be retracted to its core.

Recall that we consider the the relation $\Ab\to \Bb$ as a preorder on the
$\sigma$-structures. Keeping with this view we call a structure $\Db$ {\em
minimal} in a family $\mathcal D$ of structures if $\Db\in\mathcal D$ and for
any $\Cb\in\mathcal D$ with $\Cb\to\Db$ we have $\Db\to\Cb$. We use {\em
maximal} for the converse notion.

\begin{lemma}\label{th:nocycle-antichain}
Let $(\mathcal{O}, \mathcal{D})$ be a duality pair such that $\mathcal{D}$ is
finite and the members of $\mathcal O$ are cores. Then the minimal members of
$\mathcal{O}$ are forests.
\end{lemma}
\begin{proof}
The proof uses the concept of direct product of structures and its property
that for any three structures $\BA$, $\BB$ and $\BC$ we have
$\BA\to\BB\times\BC$ if and only if $\BA\to\BB$ and $\BA\to\BC$. See e.g.,
 \cite{NT00}. Note that this characterization is the dual of the characterization
of disjoint union: $\BA\disu\BB$ admits a homomorphism to
$\BC$ if and only if $\BA\to\BC$ and $\BB\to\BC$.

Suppose that $\BA$ is a minimal element of $\mathcal{O}$ and it is no forest.
Let $\BA_1$ be a non-tree component of $\BA$ and let
$\BA=\BA_1\disu\BA_0$. Here $\BA_0$ is empty if $\BA$ is
connected. Note that $\BA_1$, as a component of the core structure $\BA$ is
itself a core. We have $\BA_0\to\BA$ but $\BA\not\to\BA_0$ as $\BA$ is
core. If any member of $\mathcal O$ would admit a homomorphism to $\BA_0$ this
would contradict the minimality of $\BA$ in $\mathcal O$. Thus, there exist a
at least one structure $\BD\in\mathcal D$ with $\BA_0\to\BD$.

Let $\BB=\BA_1\times\BD$ be a maximal element in the non-empty finite family
$\{\BA_1\times\BD\mid\BD \in \mathcal{D}, \BA_0\to\BD\}$.
We have $\BB\to\BA_1$, but $\BA_1\to\BB$ would imply $\BA_1\to\BD$ and
together with $\BA_0\to\BD$ it would also imply $\BA\to\BD$, a contradiction.

Corollary~3.5 of~\cite{NT00} can be applied here as $\BA_1$ is connected, core
and not a tree and we have
$\BB \rightarrow \BA_1\not \rightarrow \BB$. The Corollary states that there
exists a $\sigma$-structure $\BC_1$ such that $\BB \rightarrow \BC_1\rightarrow
\BA_1$ and $\BA_1\not\to\BC_1\not\to\BB$. Let us fix such a
structure $\BC_1$ and let $\BC=\BC_1\disu\BA_0$. We have $\BC \rightarrow \BA\not \rightarrow \BC$
and since $\BA$ is minimal in $\mathcal{O}$, no member of $\mathcal{O}$ admits a
homomorphism to $\BC$. Therefore, there exists an element $\BD'$ of
$\mathcal{D}$ such that $\BC \rightarrow \BD'$. We then have $\BB \rightarrow
\BC_1\rightarrow \BA_1\times \BD'$ and $\BA_1\times \BD' \not \rightarrow \BB$
(since $\BC_1\not \rightarrow \BB$). This contradicts the maximality of $\BB$
as we also have $\BA_0\to\BD'$.
\end{proof}

\begin{lemma}
Let $(\mathcal{O}, \mathcal{D})$ be an antichain duality
such that the members of $\mathcal{O}$ are cores. Then
$|\mathcal{D}| = 1$ if and only if the members of
$\mathcal{O}$ are connected.
\end{lemma}
\begin{proof}
Suppose that the members of $\mathcal{O}$ are connected. If $\mathcal{D}$
contains two members $\BD_1$ and $\BD_2$, then $\BD_1 \cup \BD_2$
does not admit a homomorphism to any member of $\mathcal{D}$,
therefore there exists some $\BA$ in $\mathcal{O}$ which admits
a homomorphism to $\BD_1 \cup \BD_2$. Since $\BA$ is connected,
this means that $\BA$ admits a homomorphism to $\BD_1$ or $\BD_2$,
a contradiction.

Conversely, suppose that some $\BA \in \mathcal{O}$ has connected components
$\BA_1, \ldots,$ 
$\BA_n$, where $n \geq 2$. For $i = 1, \ldots, n$, put
$\BB_i = \bigcup_{j \neq i} \BA_j$. Then we have 
$\BB_i \rightarrow \BA \not \rightarrow \BB_i$,
$i = 1, \ldots, n$, therefore no member of $\mathcal{O}$ admits
a homomorphism to any $\BB_i$. Thus there is a function
$\delta: \{ 1, \ldots, n\} \rightarrow \mathcal{D}$
such that $\BB_i \rightarrow \delta(i)$. For $i \neq j$,
we must have $\delta(i) \neq \delta(j)$, otherwise $\BA \rightarrow \delta(i)$.
Therefore $|\mathcal{D}| \geq n$.
\end{proof}

\begin{theorem}\label{th:antichain}
An antichain $\mathcal{O}$ of core $\sigma$-structures has a finite
dual $\mathcal{D}$ if and only if $\mathcal{O}$ is a regular family 
of forests.
\end{theorem}

For the proof of Theorem~\ref{th:antichain} we need
that whenever the upward closure
of an antichain $\mathcal{O}$ of core forests is regular, so is $\mathcal{O}$
itself. This follows from the following more general observation.

\begin{theorem}\label{th:mincore}
The cores of the minimal elements of a regular set of forests form a
regular set.
\end{theorem}

This result is more subtle than it looks. There are regular
families $\mathcal{O}$ such that none of (a) the minimal elements in $\mathcal{O}$, (b) the
cores among the elements of $\mathcal{O}$, or (c) the cores of all elements in $\mathcal{O}$
form regular
languages. To see this, consider the type of directed graphs, consider the
family $\mathcal{O}$ consisting of the oriented paths
$P_{ij}=p(++(+-+)^i++--(-+-)^j--)$,
where a word $x_1\ldots
x_n\in\{+,-\}^n$ describes the orientation $p(x)$ of the $n$ edge path whose
$k$\/th edge is directed forward if $x_k=+$ and directed backward if
$x_k=-$.
Here $P_{ij}$ is core if and only if $i\ne
j$. The minimal elements in $\mathcal{O}$ are the
oriented paths $P_{ii}$. The core of $P_{ii}$ is $p(++(+-+)^i++)$ and these
latter oriented paths do form a
regular set as opposed to the sets mentioned in (a--c).

We start with the proof of Theorem~\ref{th:antichain} using Theorem~\ref{th:mincore} and then we will prove Theorem~\ref{th:mincore} in Section~\ref{sec:mincore}.

\begin{proof}[Proof of Theorem~\ref{th:antichain} using Theorem~\ref{th:mincore}]
The ``if'' part of the statement readily follows from
Theorem~\ref{th:dualforests}.

For the ``only if'' part assume $\mathcal{O}$ is an antichain of core
$\sigma$-structures and it has a finite dual. By
Lemma~\ref{th:nocycle-antichain} $\mathcal{O}$ must consist of forests. Applying
Corollary~\ref{lm:obreg} we obtain that the family $\mathcal{O}_{\mathcal{D}}$
of structures with no homomorphism to a member of $\mathcal D$ is regular. But
the cores of the minimal elements in this family are exactly the elements of
$\mathcal{O}$, so Theorem~\ref{th:mincore} finishes the proof.
\end{proof}

Note that in a similar fashion Theorem~\ref{th:mincore} implies the following
strengthening of Lemma~\ref{th:nocycle-antichain}: Let
$(\mathcal O,\mathcal D)$ be a duality pair with $\mathcal D$ finite and
$\mathcal O$ consisting of cores. Then the minimal elements of $\mathcal O$
form a regular family of forests.

\section{Cores of minimals of regular set of forest are regular}\label{sec:mincore}
This section is devoted to the proof of Theorem~\ref{th:mincore}.
Before proving Theorem~\ref{th:mincore} we rephrase it. 

Using category theoretic conventions we call a homomorphism $f:A\to B$ a {\em retraction} if it has a right inverse, namely a homomorphism $g:B\to A$ with $f\circ g$ being the identity on $B$. For brevity we write {\em non-retraction} for a homomorphism that is not a retraction.

The following characterization of non-retractions will be useful.

\begin{prop}\label{triv}
A homomorphism $h:\Ab\to\Bb$ is a non-retraction if and only if there is a
component $\Cb$ of $\Bb$ such that the restriction of $h$ to
no substructure $\Db\subseteq\Ab$ gives a $\Db\to\Cb$ isomorphism.
\end{prop}

\begin{proof}
If $h\circ g:\Bb\to\Bb$ is the identity and $\Cb$ is a component of
$\Bb$, then the restriction of $h$ to the substructure $\Db$ of $\Ab$
induced by the image of $\Cb$ under $g$ is an isomorphism
$\Db\to\Cb$.

Conversely, if there is such a substructure $\Db$, for each component
$\Cb$ of $\Bb$, then the inverses $g_\Cb$ of the $\Db\to\Cb$
isomorphisms give the right inverse $g:\Bb\to\Ab$ of $h$ as their union.
\end{proof}

For $\mathcal{O}\subseteq\F$ we define
\begin{eqnarray*}
\UP(\mathcal{O})&=&\{\Ab\in \F \mid \exists \Tb\in
\mathcal{O},\Tb\to \Ab\},\\
\EX(\mathcal{O})&=&\{\Ab\in \F \mid \exists \Tb\in
\mathcal{O}\mbox{ and a non-retraction }h:\Tb\to \Ab\}.
\end{eqnarray*}

\begin{prop}\label{upex}
For $\mathcal{O}\subseteq\F$ the family $\UP(\mathcal{O})\setminus\EX(\mathcal{O})$ is the set of
the cores of the minimal elements in $\mathcal{O}$.
\end{prop}

\begin{proof}
Let $\Ab\in\F$ be the core of minimal element $\Bb$ of $\mathcal{O}$. Clearly,
$\Bb\to \Ab$ ensures that $\Ab\in\UP(\mathcal{O})$. Let us consider a
homomorphism $h:\Tb\to \Ab$ with $\Tb\in \mathcal{O}$. From the minimality of
$\Bb$ and from $\Tb\to \Ab\to \Bb$, we also have $\Bb\to \Tb$ and hence also a
homomorphism $g:\Ab\to \Tb$. As $\Ab$ is a core, the homomorphism $h\circ
g:\Ab\to \Ab$ must be an isomorphism. But then $g\circ(h\circ g)^{-1}$ is a
right inverse of $h$, so $h$ is a retraction. This shows that the cores of the
minimal elements of $\mathcal{O}$ are contained in
$\UP(\mathcal{O})\setminus\EX(\mathcal{O})$.

It remains to show that each forest
$\Ab\in\UP(\mathcal{O})\setminus\EX(\mathcal{O})$ is indeed the core of a minimal element
of $\mathcal{O}$. From $\Ab\in\UP(\mathcal{O})$ we have a homomorphism $h:\Bb\to \Ab$ with
some $\Bb\in \mathcal{O}$. As $\Ab\notin\EX(\mathcal{O})$, the homomorphism $h$ is a retraction and
in particular $\Ab$ is homomorphically equivalent to $\Bb$. To show that $\Bb$ is minimal in
$\mathcal{O}$, consider an arbitrary $\Cb\in \mathcal{O}$ with $\Cb\to \Bb$ and note that $\Cb\to
\Bb\to \Ab$ implies (as above) that $\Cb$ and $\Ab$ (and hence also $\Cb$ and
$\Bb$) are homomorphically equivalent. Finally let us consider an
arbitrary homomorphism $\phi:\Ab\to \Ab$. The homomorphism $\phi\circ h:\Bb\to \Ab$
is a retraction, so we have $g:\Ab\to \Bb$ with $\phi\circ h\circ g:\Ab\to \Ab$
the identity. This implies that $\phi$ itself must be an automorphism
and thus $\Ab$ is a core. This finishes the proof of the proposition.
\end{proof}

Recall that by Lemma~\ref{lm:bprl} the difference between regular
sets is also regular. Thus Proposition~\ref{upex} above and the following
two propositions together imply Theorem~\ref{th:mincore}.

\begin{prop}\label{upreg}
For a regular set $\mathcal{O}\subseteq\F$ its upward closure $\UP(\mathcal{O})$ is also regular.
\end{prop}

\begin{proof}
By Theorem~\ref{th:dualforests} $\mathcal{O}$ has a finite dual $\mathcal D$. By
Corollary~\ref{lm:obreg} $\UP(\mathcal{O})=\mathcal O_{\mathcal D}$ is regular.
\end{proof}

Note here that a direct proof of this result
would only be simpler than the proof of Proposition~\ref{dd}
below by not having to distinguish danger points from safe points. The
proofs give a doubly exponential bound on the number of equivalence
classes of $\sim_{\UP(\mathcal{O})}$ or $\sim_{\EX(\mathcal{O})}$ in terms of the number
of equivalence classes of $\sim_\mathcal{O}$ and we believe that the number of
equivalence classes can indeed be that high for some regular sets $\mathcal{O}$.

\begin{prop}\label{dd}
For any regular set $\mathcal{O}\subseteq\F$ the set $\EX(\mathcal{O})$ is regular.
\end{prop}

\begin{proof}
For simplicity we assume that the type $\sigma$ contains no unary
relations. The proof works basically the same way in the presence of
unary relations too, but making this mild assumption makes our
presentation simpler. We start with a few definitions.

Recall that each forest has a unique decomposition as the disjoint union of
trees, its components. For a rooted forest $(\Xb,v)\in\rforests$ we denote its
rooted component by $(\Xb,v)_+$, that is $[(\Xb,v)_+]$ is the component of
$\Xb$ containing $v$ and $(\Xb,v)_+=([(\Xb,v)_+],v)\in\rtrees$. Let
$(\Xb,v)_-\in\forests$ stand for the union of the remaining components of
$\Xb$, that is $\Xb=[(\Xb,v)_+]\disu(\Xb,v)_-$.
For a tree $\Ab\in\T$ and a homomorphism $h:\Ab\to[(\Xb,v)_+]$ we define
$h^0=\{w\in V(\Ab)\mid h(w)=v\}$ to be the set of {\em root
points},  points mapped by $h$ to the root of $(\Xb,v)$.
We say that $w\in h^0$ is a {\em danger point of $h$} if a
restriction of $h$ to a suitable substructure of $\Ab$ containing $w$ is
an isomorphism to $[(\Xb,v)_+]$.
Let $h^1$ stand for the set of danger points of $h$ and note that $h^1\subseteq h^0$.
We write $h^2=h^0\setminus h^1$ stand for the {\em safe points of $h$}.
Note that we slightly abuse notation by not indicating the dependence of $h^i$
on $(\Xb,v)_+$, but this will lead to no confusion.

Let us consider $\Ab\in\T$, $(\Xb,v)\in\rforests$ and a homomorphism $h:\Ab\to[(\Xb,v)_+]$.
We define the {\em$(\Eb,\Eb_0)$-extension} of $\Ab$ for $\Eb_0\in\F$ and
$\Eb:h^0\to\rtrees$ to be the forest $\Bb$ obtained  by gluing
a copy of $\Eb(w)$ to $w$ for each $w\in h^0$ and further adding
$\Eb_0$. More precisely, we obtain $\Bb$ by taking the disjoint union of
the forests $\Ab$, $\Eb_0$ and $[\Eb(w)]$ for each $w\in h^0$ and then
identifying the root of $\Eb(w)$ with $w$ for each $w\in h^0$. We will
identify the forests $\Ab$, $\Eb_0$ and $[\Eb(w)]$ with the
corresponding substructures of $\Bb$ as long as this leads to no confusion.

Let us fix the regular set $\mathcal{O}\subseteq\F$.
We call a map $L: \rforests/\!\!\sim_{\mathcal O}\to\{0,1,2\}$ a {\em list}.

Let $\Ab\in\T$, $(\Xb,v)\in\rforests$, $h:\Ab\to [(\Xb,v)_+]$ and let $\Bb$ be the
$(\Eb,\Eb_0)$-extension of $\Ab$. We define the {\em signature} of this extension
to be the triple $(L_1,L_2,{\mathcal C}')$, where ${\mathcal C}'=\Eb_0/\!\!\approx_{\mathcal O}\in \forests/\!\!\approx_{\mathcal O}$ and $L_1$
and $L_2$ are lists defined by $L_i({\mathcal C})=\min(2,|\{w\in
h^i\mid\Eb(w)\in{\mathcal C}\}|)$ for $i=1,2$ and ${\mathcal C}\in \rforests/\!\!\sim_{\mathcal O}$.

We say that $(\Xb,v)\in\rforests$ is {\em compatible} with the signature $T$ if there is
$\Ab\in\T$, $h:\Ab\to [(\Xb,v)_+]$ and an extension $\Bb$ of $\Ab$ of signature $T$ with
$\Bb\in \mathcal{O}$. We say that $(\Xb,v)\in\rforests$ is {\em compatible}
with ${\mathcal C}'\in \forests/\!\!\approx_{\mathcal O}$ if there is a forest $\Ab\in{\mathcal C}'$ satisfying
$\Ab\to (\Xb,v)_-$. If a non-retraction $\Ab\to (\Xb,v)_-$ also exists from
such a forest $\Ab\in{\mathcal C}'$, we say that $(\Xb,v)$ is {\em strongly compatible}
with ${\mathcal C}'$.

Note that since $\mathcal{O}$ is regular, $\rforests/\!\!\sim_{\mathcal O}$ and $\forests/\!\!\approx_{\mathcal O}$ must be finite (the
latter by Lemma~\ref{lem:tree-eq}), there is a finite number of different
signatures. Let us accept Proposition~\ref{new} below. It implies that
$\sim_{\EX(\mathcal{O})}$ has a finite number equivalence classes, in other
words, that $\EX(\mathcal{O})$ is regular, finishing the proof of
Proposition~\ref{dd}.
\end{proof}

To complete the proof above it remains to prove the following:

\begin{prop}\label{new} If $(\Xb,v),(\Xb',v)\in\rforests$ are compatible with
  the same signatures and the same elements of $\forests/\!\!\approx_{\mathcal O}$ and they are also strongly
  compatible with the same elements of $\forests/\!\!\approx_{\mathcal O}$ and further $(\Xb,v)_+=\Tb_0$ if
  and only if $(\Xb',v)_+=\Tb_0$, then $(\Xb,v)\sim_{\EX(\mathcal{O})}
  (\Xb',v)$.
\end{prop}

\begin{proof}
Let $(\Xb,v),(\Xb',v)$ satisfy the condition of the
proposition and $(\Yb,v)\in\rforests$ satisfy $[(\Xb,v)+(\Yb,v)]\in\EX(\mathcal{O})$. By symmetry it is enough to prove that $[(\Xb',v)+(\Yb,v)]\in\EX(\mathcal{O})$.

By the definition of $\EX(\mathcal{O})$ we have a non-retraction
$h:\Bb\to[(\Xb,v)+(\Yb,v)]$ for some $\Bb\in \mathcal{O}$. Our goal is to find another 
non-retraction $h':\Bb'\to[(\Xb',v)+(\Yb,v)]$ with $\Bb'\in \mathcal{O}$.
We do the transformation step by step.
In each step we have some $\Bb_i\in \mathcal{O}$ and a homomorphism
$h_i:\Bb_i\to[(\Xb,v)_++(\Xb',v)+(\Yb,v)]$ using ``less and less'' the part of $[(\Xb,v)_++(\Xb',v)+(\Yb,v)]$ coming from $(\Xb,v)_+$.

Consider the substructure $\Ab_0$ of $\Bb$ induced by the vertices $h$
maps to $(\Xb,v)_-$. Clearly, $(\Xb,v)$ is compatible with ${\mathcal C}'=\Ab_0/\!\!\approx$. If $h$
restricted to $\Ab_0$ is a non-retraction $\Ab_0\to (\Xb,v)_-$, then $(\Xb,v)$ is
strongly compatible with ${\mathcal C}'$. So $(\Xb',v)$ must also be compatible with ${\mathcal C}'$
and, in the latter case, also strongly compatible with ${\mathcal C}'$. Let
$\Ab'_0\in{\mathcal C}'$ and the homomorphism $g'_0:\Ab'_0\to (\Xb',v)_-$ show this, that is $g'_0$ is a non-retraction if $(\Xb,v)$ is strongly compatible with ${\mathcal C}'$.
Let $\Bb_1=\Ab'_0\disu\Bb^*$, where $\Bb^*$ is the
substructure of $\Bb$ outside $\Ab_0$. From $\Bb=\Ab_0+\Bb^*\in \mathcal{O}$
and $\Ab\approx_{\mathcal O}\Ab'$ we have
$\Bb_1\in \mathcal{O}$. We
define $h_1:\Bb_1\to[(\Xb,v)_++(\Xb',v)+(\Yb,v)]$ by
making its restriction to $\Ab'_0$ be $g'_0$ (this maps to $(\Xb',v)_-$)
and making its restriction
to $\Bb^*$ be the same as the corresponding restriction of $h$ (mapping
to $[(\Xb,v)_++(\Yb,v)]$).

Next we give the recursive step. Assume $\Bb_i\in \mathcal{O}$ and
$h_i:\Bb_i\to[(\Xb,v)_++(\Xb',v)+(\Yb,v)]$ is given.
We define the {\em old-parts} of $\Bb_i$ to be the maximal connected
substructures of $\Bb_i$ that $h_i$ maps
into $[(\Xb,v)_+]$ (considered as a substructure of
$[(\Xb,v)_++(\Xb',v)+(\Yb,v)]$). We exclude single vertex substructures mapped
to $v$ and do not consider these old-parts.
We measure progress by the decreasing number of old-parts, that is, we make sure
that $\Bb_{i+1}$ has fewer old-parts than $\Bb_i$.
This ensures that the procedure terminates with no old-parts left.
If $\Bb_i$ has no old-parts, then $h_i$ maps the entire structure $\Bb_i$ to
$[(\Xb',v)+(\Yb,v)]$. In this case we set $\Bb'=\Bb_i$ and $h'=h_i$.

Assume now that there is still at least one old-part of $\Bb_i$. As
$\ig(\Bb_i)$ is a forest with the old-parts being pairwise disjoint
subtrees we can choose an old-part $\Ab_i$ that does not separate two
further old-parts in this graph. Let us fix such an old-part $\Ab_i$ and let
$g_i:\Ab_i\to[(\Xb,v)_+]$ denote the restriction of $h_i$ to $\Ab_i$. There is a unique
way to express $\Bb_i$ as an $(\Eb_i,\Eb_{i,0})$-extension of $\Ab_i$: we set $\Eb_{i,0}$ to
be the forest consisting of the components of $\Bb_i$ other
than the component $\Ab_i^*$ containing $\Ab_i$, while for $w\in g_i^0$ we set
$\Eb_i(w)$ to be the substructure of $\Ab_i^*$ that is separated from $\Ab_i$ by
$w$ in the tree $\ig(\Ab^*)$. We make $\Eb_i(w)$ include $w$ as its
root. Note that for $w\in V(\Ab_i)\setminus g_i^0$ the image $g_i(w)$ of $w$
is a not $v$, so such a $w$ cannot separate the old-part $\Ab_i$
from any points in $\Ab_i^*$ as otherwise $\Ab_i$ would not be a maximal
connected substructure mapped to $[(\Xb,v)_+]$.

By our choice of $\Ab_i$, if there are further old-parts in $\Ab_i^*$
they must all be contained in a single substructure $[\Eb_i(w)]$. In this case
we denote the corresponding vertex $w\in g_i^0$ as $w_i$. In case $\Ab_i$ is
the only old-part of $\Bb_i$ no vertex $w_i$ is specified.

Let $T_i$ be the signature of the $(\Eb_i,\Eb_{i,0})$-extension $\Bb_i$ of
$\Ab_i$. Now $\Ab_i\in\T$,  $g_i:\Ab\to[(\Xb,v)_+]$ and $\Bb_i\in\mathcal O$
shows that $(\Xb,v)$ is compatible with $T_i$.
By our assumption $(\Xb',v)$ must also be compatible with the signature
$T_i$. Let $\Ab'_i\in\T$, $g_i':\Ab'\to[(\Xb',v)_+]$ and its
$(\Eb_i',\Eb'_{i,0})$-extension
$\Bb_i'\in \mathcal{O}$ show this, that is, assume this extension has
signature $T_i$.

It is tempting at this point to define $\Bb_{i+1}=\Bb_i'$ but then we would
need to find a homomorphism $h_{i+1}:\Bb'_i\to[(\Xb,v)_++(\Xb',v)+(\Yb,v)]$,
which may not exist as some of
the forests $[\Eb_i'(w)]$ or $\Eb_{i,0}'$ may not have a homomorphism to
$[(\Xb,v)_++(\Xb',v)+(\Yb,v)]$. We need to replace the rooted trees
$\Eb_i'(w)$ with a $\sim_{\mathcal O}$-equivalent rooted tree $\Eb(w')$ first
to make the homomorphism possible.

Let $T_i=(L_{i,1},L_{i,2},{\mathcal C}'_i)$. We define $f_i:g_i^{\prime0}\to g_i^0$ as
follows. For each $j=1,2$ and $w\in g_i^{\prime j}$ we choose $f_i(w)\in
g_i^j$ with $\Eb_i(f_i(w))\sim_\mathcal{O} \Eb'_i(w)$.
The number of possible choices for $w'=f_i(w)$ is $|\{w'\in
g_i^j\mid\Eb_i(w')\sim_{\mathcal O}\Eb'_i(w)\}|\ge
L_{i,j}(\Eb'_i(w)/\!\!\sim_{\mathcal O})\ge1$, so such a choice is always
available. If there are old-parts in $\Ab_i^*$ beyond $\Ab_i$, we pick
$f_i(w)=w_i$ for at most a single $w\in g_i^{\prime0}$. This is also possible
because if we have more than one $w$ for which $f_i(w)=w_i$ is possible at
all, then $L_{i,j}(\Eb_i(w_i)/\!\!\sim_{\mathcal O})=2$ for the corresponding
value of $j=1$ or $2$, so we have the freedom not to choose $f_i(w)=w_i$ for
any $w$.

We set $\Eb''_i=\Eb_i\circ f_i$, that is for $w\in g_i^{\prime0}$ we have
$\Eb''_i(w)=\Eb_i(f_i(w))$. We set $\Bb_{i+1}$ to be the
$(\Eb_i'',\Eb_{i,0})$-extension of $\Ab'_i$. (Note here that $f_i(w)=f_i(w')$
for distinct root vertices of $g_i'$ is possible as long as $f_i(w)\ne
w_i$. In this case we have $\Eb''_i(w)=\Eb''_i(w')=\Eb_i(f_i(w))$, and the two
substructures $[\Eb''_i(w)]$ and $[\Eb''_i(w')]$ are isomorphic, but naturally
they are disjoint substructures. This shows the limits of
our notation, but hopefully leads to no confusion.)
We define $h_{i+1}:\Bb_{i+1}\to[(\Xb,v)_++(\Xb',v)+(\Yb,v)]$ through its
restrictions. The restriction to $\Ab_i'$ is $g_i'$, the restriction to
$\Eb_{i,0}$ is the restriction of $h_i$ to $\Eb_{i,0}$. Finally, for $w\in
g_i^{\prime0}$ the restriction of $h_{i+1}$ to $[\Eb''_i(w)]$ is the
restriction of $h_i$ to $[\Eb_i(f_i(w))]$.
These restrictions uniquely define the homomorphism $h_{i+1}$ as the
given substructures cover $\Bb_{i+1}$, only the root vertices in
$g_i^{\prime0}$ are covered more than once and these
points are mapped to $v$ in all the given
restrictions.

For the recursive definition to work we need to show that $\Bb_{i+1}$ is
in $\mathcal{O}$ and it has fewer old-parts than $\Bb_i$.

We start with showing that $\Bb_{i+1}\in \mathcal{O}$. Note that
$\Eb_{i,0}\approx_\mathcal{O} \Eb_{i,0}'$ since both of them are contained in the $\approx_{\mathcal O}$
equivalence class $\BC'_i$ (part of the signature $T_i$). As the
$(\Eb_i',\Eb'_{i,0})$-extension of $\Ab'_i$ is $\Bb'_i\in\mathcal{O}$ the
$(\Eb_i',\Eb_{i,0})$-extension of $\Ab_i'$ must also be in $\mathcal O$. The
structure $\Bb_{i+1}$ differs from this last structure by being obtained as an
extension using the function $\Eb_i''$ instead of $\Eb_i'$. As we always have
$\Eb'_i(w)\sim_\mathcal{O} \Eb''_i(w)$, a similar argument applies: a single
such change does not alter membership in $\mathcal{O}$. Doing these changes
one by one we obtain eventually that $\Bb_{i+1}\in \mathcal{O}$ as claimed.

Consider now the old-parts of $\Bb_{i+1}$. Some may be found in $\Eb_{i,0}$,
but these are also old-parts of $\Bb_i$. Others may be found in some
$[\Eb''_i(w)]$ but only if $[\Eb_i(f_i(w))]$ contains an old-part of
$\Bb_i$. By our assumption this holds only for at most a single $w\in
g_i^{\prime0}$ with $f_i(w)=w_i$, so we have no more old-parts in $\Bb_{i+1}$
than in $\Bb_i$. In fact, we have fewer as $\Ab_i$ was an old-part in $\Bb_i$
and it got replaced by $\Ab'_i$ that is mapped to $(\Xb',v)$.

We have defined the homomorphism $h':\Bb'\to[(\Xb',v)+(\Yb,v)]$. Before
finishing the proof of Proposition~\ref{new} by showing that this is a 
non-retraction we make an easy observation. We call two distinct points
in the universe of a $\sigma$-structure {\em neighbors} if they appear in a
common block, that is, if they are in distance two in the incidence
graph. For any step $i$ in the procedure above and any $w\in
g_i^{\prime0}$ the homomorphism $h_{i+1}$ maps all neighbors of $w$ in
$[\Eb_i''(w)]\subseteq \Bb_{i+1}$ to non-root points of $(\Yb,v)$. This is
shown by an easy induction together with the statement that no point in
$\Bb_i$ has both a neighbor that $h_i$ maps to a non-root point of $(\Xb,v)_+$ and
another neighbor that $h_i$ maps to a non-root point of $(\Xb',v)_+$.

Now we turn to the proof of $h'$ being a non-retraction. By
Proposition~\ref{triv} and since $h:\Bb\to[(\Xb,v)+(\Yb,v)]$ is a non-retraction,
$[(\Xb,v)+(\Yb,v)]$ has a component $\Zb$ with no restriction of $h$ being an
isomorphism to $\Zb$.

If $\Zb\subseteq[(\Yb,v)]$, then the part of $h_1$ that maps to $\Zb$ is copied
from $h$, later the part of $h_{i+1}$ that maps to $\Zb$ is also copied
from $h_i$, so as no restriction $h$ is an isomorphism to $\Zb$ the same
can be said about the last function $h'$. Here $\Zb$ is a component of
$[(\Xb',v)+(\Yb,v)]$ either because $\Zb\subseteq (\Yb,v)_-$ or because $\Zb=[(\Yb,v)_+]$ and
$(\Xb,v)_+=\Tb_0$, in the latter case forcing $(\Xb',v)=\Tb_0$. Again by
Proposition~\ref{triv} no restriction is an isomorphism to the component $\Zb$
means that $h'$ is a non-retraction as claimed.

If $\Zb\subseteq (\Xb,v)_-$, then $h$ restricted to $\Ab_0$ is a
non-retraction to $(\Xb,v)_-$ as it has no restriction that is an isomorphism to
$\Zb$, so by the choice of $g'_0:\Ab'_0\to (\Xb',v)_-$ it is also a non-retraction
and we have a component
$\Zb'$ of $(\Xb',v)_-$ such that no restriction of $g'_0$ is an isomorphism to
$\Zb'$. This means that no restriction of $h_1$ is an isomorphism to
$\Zb'$ and, since the part mapping to $(\Xb',v)_-$ is copied from $h_i$ to
$h_{i+1}$, we conclude that no restriction of $h'$ is an
isomorphism to $\Zb'$ and therefore $h'$ is a non-retraction.

Finally we consider the $\Zb=[(\Xb,v)_++(\Yb,v)_+]$ case with $(\Xb,v)_+$ (and thus also $(\Xb',v)$) not being $\Tb_0$. Assume for a contradiction that $h'$ is a retraction.
By Proposition~\ref{triv} we have a restriction of $h'$ that
is an isomorphism to $\Zb'=[(\Xb',v)_++(\Yb,v)_+]$. No restriction of $h_1$
is an isomorphism to either $\Zb$ or $\Zb'$, the former because the
relevant part of $h_1$ is copied from $h$, the latter because the non-root vertices
of $(\Xb',v)_+$ do not even appear in the image of $h_1$. Thus we must have
an index $i$ with no restriction of $h_i$ being an isomorphism to
either $\Zb$ or $\Zb'$ but such that the restriction of $h_{i+1}$ to a
substructure $\Db$ of $\Bb_{i+1}$ is suddenly an isomorphism to one of $\Zb$
or $\Zb'$. We cannot have $\Db\subseteq \Eb_{i,0}$ or $\Db\subseteq[\Eb''_i(w)]$
for some $w\in g_i^{\prime0}$ as then the
restriction of $h_i$ to the corresponding substructure in $\Bb_i$ would
also be an isomorphism to $\Zb$ or $\Zb'$. The substructure $\Db$ is
connected and it contains a single point $w\in V(\Db)$ mapped to the
root $v$, so we must have $w\in g_i^{\prime0}$ and $\Db$
is the union of its substructures $\Db_0$ contained in $\Ab'_i$ and
$\Db_1$ contained in $[\Eb''_i(w)]$. Here $\Db_0$ is not trivial
and mapped to $[(\Xb',v)_+]$, so $h_{i+1}$ must map $\Db$ isomorphically to
$\Zb'$. We know that the neighbors of $w$ in $\Eb_i''(w)$ are
mapped by $h_{i+1}$ to non-root vertices of $(\Yb,v)$. As no point in $V(\Db)$
other than $w$ is mapped to the root we must have that $h_{i+1}$ maps $\Db_1$
to $[(\Yb,v)]$. To make the restriction of $h_{i+1}$ to $\Db$ an isomorphism
to $\Zb'$ we must have that $\Db_0$ is mapped isomorphically to
$[(\Xb,v)_+]$ and $\Db_1$ is mapped isomorphically to $[(\Yb,v)_+]$. The restriction to
$\Db_0$ makes $w$ a danger point of $g_i'$, i.e., $w\in g^{\prime1}_i$. Thus,
we also have $f_i(w)\in g_i^1$ is a danger point of $g_i$. Let $\Db'\subseteq
\Ab_i$ be the substructure showing this, that is $\Db'$ contains $f_i(w)$ and
$g_i$ maps $\Db'$ isomorphically to $[(\Xb,v)_+]$. Now $h_i$ maps the union of $\Db'$
with the substructure corresponding to $\Db_1\subseteq[\Eb''_i(w)]$ in
$[\Eb_i(f_i(w))]$ isomorphically to $\Zb$. This contradicts our
assumptions and proves that $h':\Bb'\to[(\Xb',v)+(\Yb,v)]$ is a non-retraction.

Proving that $h'$ is a non-retraction finishes the proof of Proposition~\ref{new} and
thus also completes the proof of Proposition~\ref{dd}, Theorem~\ref{th:mincore} and Theorem~\ref{th:antichain}.
\end{proof}

\subsubsection*{Acknowledgment} 
We would like to thank an anonymous referee for several invaluable suggestions.
We would also like to thank Miko\l aj Boja\'nczyk for calling our attention to \cite{MB} where other variants of regular tree languages are discussed.

\bibliographystyle{plain}

\end{document}